\documentclass{article}
\usepackage{amsmath,amsfonts,latexsym,amssymb,amsthm,url}
\usepackage{amsmath,amssymb,amsbsy,amsfonts,amsthm,latexsym,
            amsopn,amstext,amsxtra,euscript,amscd,amsthm}
\usepackage[T1]{fontenc}
\usepackage{lmodern,enumitem}
\usepackage{stackengine,scalerel}
\usepackage[toc]{appendix}
\usepackage[utf8]{inputenc}
\usepackage{xcolor}

\newcommand\Humlaut[1]{\stackengine{-.05ex}{#1}{\hstretch{.8}{\vstretch{.65}{%
  \mkern1mu\scriptscriptstyle''}}}{O}{c}{F}{F}{S}}

\makeatletter
\let\@fnsymbol\@alph
\makeatother

\newcommand{\K}{{\mathbb K}}
\newcommand{\Q}{{\mathbb Q}}

\newcommand{\Z}{{\mathbb Z}}

\newcommand{\calA}{\mathcal{A}}
\newcommand{\calD}{\mathcal{D}}

\newcommand{\calO}{\mathcal{O}}
\newcommand{\calQ}{\mathcal{Q}}
\newcommand{\calZ}{\mathcal{Z}}

\newcommand\eps\varepsilon
\newcommand\ph\varphi

\newcommand{\Gal}{\operatorname{Gal}}

\newtheorem{theorem}{Theorem}[section]
\newtheorem{proposition}[theorem]{Proposition}
\newtheorem{corollary}[theorem]{Corollary}
\newtheorem{lemma}[theorem]{Lemma}
\newtheorem{remark}[theorem]{Remark}

\newtheorem{conjecture}[theorem]{Conjecture}
\numberwithin{equation}{section}

\newcounter{jump}

\title{On the $p$-adic zeros of the Tribonacci sequence}

\author{Yuri Bilu\\
IMB, Université de Bordeaux \& CNRS\\
E-mail: yuri@math.u-bordeaux.fr,
\and
Florian Luca\\
School of Maths, Wits, South Africa\\
Research Group in Algebraic Structures and Applications,\\ 
King Abdulaziz University, Saudi Arabia\\
CCM, UNAM, Morelia, Mexico\\
E-mail: Florian.Luca@wits.ac.za,
\and
Joris Nieuwveld and J\"oel Ouaknine\\
Max Planck Institut for Software Systems, Germany\\
E:mail jnieuwve@mpi-sws.org; joel@mpi-sws.org,
\and
James Worrell\\
Department of Computer Science, Oxford University, UK\\
jbw@cs.ox.ac.uk
}

\makeatletter

\renewcommand*\l@section[2]{%
  \ifnum \c@tocdepth >\z@
    \addpenalty\@secpenalty
    \addvspace{0.2em \@plus\p@}%
    \setlength\@tempdima{1.5em}%
    \begingroup
      \parindent \z@ \rightskip \@pnumwidth
      \parfillskip -\@pnumwidth
      \leavevmode \bfseries
      \advance\leftskip\@tempdima
      \hskip -\leftskip
      #1\nobreak\hfil \nobreak\hb@xt@\@pnumwidth{\hss #2}\par
    \endgroup
  \fi}

\makeatother

\begin{document}

\hfuzz 6pt


\maketitle

{\footnotesize

\tableofcontents

}

\section{Introduction}
Let ${\Lambda=\{\lambda_1,\lambda_2,\lambda_3\}\subset \overline\Q}$ be the set of roots of the polynomial 
$$
P(X)=X^3-X^2-X-1.
$$
For ${\lambda \in \Lambda}$ define ${c_\lambda=\lambda P'(\lambda)^{-1}}$.  
For ${n\in \Z}$, the \textit{Tribonacci number} ${T(n)\in \Z}$ is defined by 
$$
T(n)=\sum_{\lambda \in \Lambda}c_\lambda\lambda^n. 
$$
More famously, ${T:\Z\to \Z}$ is defined by the recurrence relation 
\begin{align*}
&T(0)=0,\qquad  T(1)=T(2)=1, \\ 
&T(n+3)=T(n+2)+T(n+1)+T(n) \qquad (n\in \Z). 
\end{align*}
Attention: $a(n)$ in~\cite{oeis} corresponds to our ${T(n+1)}$.

It is known  that  ${T(n)=0}$ if and only if ${n\in {\mathcal Z}_T:=\{0,-1,-4,-17\}}$. For a proof see,  for instance,~\cite{MT91}, Example~2 on page~360; in that example~$u_n$  corresponds to our $T(-n)$. In \cite{ML14}, Marques and Lengyel determined 
the exponent of $2$ in $T_n$. Denoting for a prime $p$ and an nonzero integer $m$ by $\nu_p(m)$ the exact exponent of $p$ in the factorization of $m$, and setting $\nu_p(0)=\infty$,
they proved the following theorem. 

\begin{theorem}
For $n\ge 1$, we have
$$
\nu_2(T_n)=\left\{\begin{matrix} 0, & {\text{\rm if}} & n\equiv 1,2\pmod 4;\\
1, & {\text{\rm if}} & n\equiv 3,11\pmod {16};\\
2, & {\text{\rm if}} & n\equiv 4,8\pmod {16};\\
3, &{\text{\rm if}} & n\equiv 7\pmod {16};\\
\nu_2(n)-1, & {\text{\rm if}} & n\equiv 0\pmod {16};\\
\nu_2(n+4)-1, &{\text{\rm if}} & n\equiv 12\pmod {16};\\
\nu_2(n+17)+1, &{\text{\rm if}} & n\equiv 15\pmod {32};\\
\nu_2(n+1)+1, &{\text{\rm if}} & n\equiv 31\pmod {32}.
\end{matrix}
\right.
$$
\end{theorem}
Encouraged by their result for the prime $p=2$, they conjectured that such formulas should hold for $\nu_p(T_n)$ for every prime $p$. More precisely, here is their conjecture.

\begin{conjecture}[Conjecture~8 from~\cite{ML14}]
\label{conml}
Let~$p$ be a prime number.
There exists a positive integer~$Q$ such that for every ${i\in \{0,1, \ldots, Q-1\}}$ we have one of the following two options.
\begin{enumerate}
\item[(C)]
There exists ${\kappa_i \in \Z_{\ge 0}}$ such that for all but finitely many ${n\in \Z}$ satisfying ${n\equiv i \pmod Q}$ we have ${\nu_p(T(n))=\kappa_i}$. 

\item[(L)]
There exist 
$$
a_i\in \Z, \qquad \kappa_i\in \Z, \qquad  \mu_i \in \Z_{>0}
$$
satisfying 
\begin{equation}
\label{enup}
\nu_p(a_i-i) \ge \nu_p(Q),
\end{equation}
such that for all but finitely many ${n\in \Z}$ satisfying ${n\equiv i \pmod Q}$ we have
\begin{equation}
\label{ehrs}
\nu_p(T(n)) =\kappa_i+\mu_i\nu_p(n-a_i). 
\end{equation} 
\end{enumerate}
\end{conjecture}
Note that our statement looks different from Conjecture~8 from~\cite{ML14}, but, in fact, it is equivalent to it.

Informally, in the case~(C) (that is, ``constant'') $\nu_p(T(n))$ is a constant function on the entire residue class ${n\equiv i \pmod Q}$ with finitely many~$n$ removed, while in the case~(L) (``linear'') it is a linear function of ${\nu_p(n-a_i)}$. 

\begin{remark}
\label{rimp}
Let us comment on condition~\eqref{enup}, which does not appear in~\cite{ML14}. This condition is needed to ensure that the right-hand side of~\eqref{ehrs} is not constant (in which case option~(C) would hold for the class ${n\equiv i\pmod Q}$). To be precise, the following three statements are equivalent: 
\begin{enumerate}
\item
\label{ih}
\eqref{enup} holds; 
\item
\label{inc}
$\nu_p(n-a_i)$ is not constant on the residue class ${n\equiv i\pmod Q}$;
\item
\label{inb}
$\nu_p(n-a_i)$ is not bounded on the residue class ${n\equiv i\pmod Q}$.
\end{enumerate} 
Indeed, if ${\nu_p(a_i-i) < \nu_p(Q)}$ then  ${\nu_p(n-a_i)=\nu_p(i-a_i)}$ for  ${n\equiv i \pmod Q}$, which proves the implication \ref{inc}.$\Rightarrow$\ref{ih}. The implication \ref{inb}.$\Rightarrow$\ref{inc}.\ is obvious. Finally, assume that~\eqref{enup} holds. Denoting ${\nu:=\nu_p(Q)}$, for every ${k\ge \nu}$ the Chinese remainder theorem provides ${m_k\in \Z}$ satisfying
$$
m_k\equiv \frac{a_i-i}{p^\nu} \pmod{p^{k-\nu}}, \qquad m_k \equiv 0 \pmod {Qp^{-\nu}}. 
$$
Then ${n_k:=i+m_kp^\nu}$ satisfies 
${n_k\equiv a_i \pmod{p^{k}}}$ and ${n_k \equiv i \pmod {Q}}$,  
which proves the implication \ref{ih}.$\Rightarrow$\ref{inb}. 
\end{remark}

Already the case $p=3$ looks encouraging.

\begin{theorem}
\label{thm:p=3}
For $n\ge 1$, we have
$$
\nu_3(T_n)=\left\{\begin{matrix} 
0, &{\text{\rm if}} & n\equiv 1,2,3,4,5,6,8,10,11\pmod {13};\\
1, & {\text{\rm if}} & n\equiv 7\pmod {13};\\
\nu_3(n)+2, & {\text{\rm if}}  & n\equiv 0\pmod {13};\\
\nu_3(n+1)+2, & {\text{\rm if}} & n\equiv 12\pmod {13};\\
4, &{\text{\rm if}} & n\equiv 9\pmod {39};\\
\nu_3(n+17)+4, & {\text{\rm if}} & n\equiv 22\pmod {39};\\
\nu_3(n+4)+4, &{\text{\rm if}} & n\equiv 35\pmod {39}.
\end{matrix}
\right.
$$
\end{theorem}

However, the following theorem shows that Conjecture~\ref{conml} fails for infinitely many primes.

\begin{theorem}
\label{ththird}
There is an infinite set\footnote{We will see that this set of primes is not only infinite, but is of relative density $1/12$ in the set of all primes.} of prime numbers congruent to ${2\pmod 3}$   such that for every prime~$p$ from this set the following holds. 
\begin{enumerate}
\item
For each ${n\in \Z}$ satisfying ${n\equiv 1/3\pmod {p-1}}$ we have 
$$
\nu_p(T(n)) \ge\nu_p(n-1/3). 
$$

\item
For each ${n\in \Z}$ with ${n\equiv -5/3\pmod {p-1}}$ we have 
$$
\nu_p(T(n)) \ge \nu_p(n+5/3).
$$ 
\end{enumerate}
\end{theorem}

Clearly, Theorem~\ref{ththird} contradicts Conjecture~\ref{conml}. Indeed, let~$p$ be as in the theorem, and let $(n_k)$ be a sequence of integers satisfying 
$$
n_k\equiv 1/3\pmod{(p-1)p^k}.
$$
If Conjecture~\ref{conml} is true for this~$p$ then for some ${i\in \{0,\ldots, Q-1\}}$ the residue class ${i\pmod Q}$ contains infinitely many~$n_k$. Since ${\nu_p(n_k-1/3)\to \infty}$,  we have ${\nu_p(T(n))\to \infty}$. Hence for this~$i$ we must have option~(L) of Conjecture~\ref{conml}:
$$
\nu_p(T(n_k))= \kappa_i+\mu_i\nu_p(n_k-a_i). 
$$
Moreover, we must have ${\nu_p(n_k-a_i)\to \infty }$ as well. But,
since ${a_i\in \Z}$, we have ${a_i\ne 1/3}$, and hence  
${\nu_p(n_k-1/3)}$ and  ${\nu_p(n_k-a_i)}$ cannot both tend to infinity. 

One may hope to rescue Conjecture~\ref{conml} by allowing~$a_i$ to be
rational numbers, as below:

\begin{conjecture}
\label{conmlrat}
Let~$p$ be a prime number.
There exists a positive integer~$Q$ such that for every ${i\in \{0,1, \ldots, Q-1\}}$ we have one of the following two options.
\begin{enumerate}
\item[(C)]
There exists ${\kappa_i \in \Z_{\ge 0}}$ such that for all but finitely many ${n\in \Z}$ satisfying ${n\equiv i \pmod Q}$ we have ${\nu_p(T(n))=\kappa_i}$. 

\item[(L)]
There exist 
$$
a_i\in \Q, \qquad \kappa_i\in \Z, \qquad  \mu_i \in \Z_{>0}
$$
satisfying ${\nu_p(a_i-i) \ge \nu_p(Q)}$, such that for all but finitely many ${n\in \Z}$ satisfying ${n\equiv i \pmod Q}$ we have
${\nu_p(T(n)) =\kappa_i+\mu_i\nu_p(n-a_i)}$. 

\end{enumerate}
\end{conjecture}

However we show that even this  weaker conjecture fails for many primes. 
In fact we provide a method to decide for which
primes~$p$ Conjectures~\ref{conml} and~\ref{conmlrat} hold and for
which they fail.  In some cases our method is unable to
make the desired decision. When the method works and decides that the
conjecture holds, it also determines the parameters $Q$ and
$(a_i,\mu_i)$ for those $i=\{0,\ldots,Q-1\}$ for which option~(L)
takes place.

Concerning Conjecture~\ref{conml}, we have:
\begin{theorem}
\label{thm5to283}
\begin{itemize}
\item[(i)] Conjecture \ref{conml} fails for $p\in [5,599]\backslash \{11,83,103,163,397\}$. 
\item[(ii)] Conjecture \ref{conml} holds for $p\in\{83, 397\}$ in the form 
$$
\nu_p(T_n)=\left\{\begin{matrix} \nu_p(n+c)+1, & {\text{\rm if}} & n\equiv -c\pmod {Q_p},~-c\in {\mathcal Z}_T;\\
0, & {\text{\rm otherwise,}}  & ~~~\\
\end{matrix}
\right. 
$$
with $Q_{83}=287$ and $Q_{397}=132$. 
\end{itemize} 
\end{theorem}
Note that our method does not handle the prime $p=11$. As for
$p\in \{103,163\}$, our method failed to decide whether
Conjecture~\ref{conml} holds.

Concerning Conjecture~\ref{conmlrat}, we have:
\begin{theorem}
\label{thm5to283rat}
\begin{itemize}
\item[(i)] Conjecture \ref{conmlrat} fails for $$p\in [5,599]\backslash \{11,47,53,83,103,163,269,397,401,419,499,587\}.$$
\item[(ii)] Conjecture \ref{conmlrat} holds for $p\in\{269, 401,419,499,587\}$ in the form
$$
\nu_p(T_n)=\left\{\begin{matrix} \nu_p(n+c)+1, & {\text{\rm if}} & n\equiv c\pmod {Q_p},\\
&& c\in \{0,-1,-4,-17,1/3,-5/3\};\\
0, & {\text{\rm otherwise,}}  & ~~~\\
\end{matrix}
\right.
$$
with $Q_{269}=268$, $Q_{401}=400$, $Q_{419}=418$, $Q_{499}=166$ and $Q_{587}=293$.
\end{itemize} 
\end{theorem}
Note here that (again) our method does not apply to the prime
$p=11$. As for $p \in \{47,53,103,163\}$, our method failed to decide
whether Conjecture~\ref{conmlrat} holds.

\paragraph{Plan of the article}
In Section~\ref{srat} we introduce the basic notions of this article, those of twisted zeros and of rational zeros of the Tribonacci sequence.

In Section~\ref{spadic}  we recall the necessary tools from $p$-adic analysis. In Sections~\ref{sfunction}--\ref{sex} we apply these tools to study the Tribonacci sequence. In particular, Theorem~\ref{ththird} is proved in Section~\ref{sthfive}. In Section~\ref{sstatements} we give a $p$-adic analytic interpretation of Conjectures~\ref{conml} and~\ref{conmlrat}. Using it, we produce in Section~\ref{sex} easily verifiable sufficient conditions for both conjectures to hold and to fail. 

Theorems~\ref{thm:p=3} and~\ref{thm5to283} are proved in Section~\ref{sec:comput}, as application of the previous results together 
with some computations. 

The final Section~\ref{sheur} contains heuristics which suggest that if ${\mathcal ML}$ and ${\mathcal NMLR}$ are the sets of all primes such that Conjecture \ref{conml} holds and Conjecture \ref{conmlrat} and fails, respectively, then both ${\mathcal ML}$ and ${\mathcal NMLR}$ are infinite and maybe even of positive relative densities as subsets of all primes.

\paragraph{A convention}
Unless otherwise stated, all congruences such as
${x\equiv y \pmod N}$ and divisibility relations such as ${x\mid y}$
refer to the ring of rational integers~$\Z$.

We slightly abuse notation by writing ${x\equiv y \pmod N}$ with ${x,y\in \Q}$ if there exists ${m\in \Z}$ with ${\gcd(m,N)=1}$ such that ${mx,my\in \Z}$ and ${mx\equiv my \pmod N}$.

\section{Rational zeros of the Tribonacci sequence}
\label{srat}

As we mentioned in the introduction, ${T(n)=0}$ if and only if~$n$ belongs to the set  ${\calZ_T=\{0,-1,-4,-17\}}$. It turns out that, in a sense, the Tribonacci sequence also ``vanishes'' at some non-integral rational numbers.

\begin{proposition}
For some definition of the cubic roots 
\begin{equation}
\label{ecubicroots}
\lambda^{1/3} \qquad (\lambda\in \Lambda)
\end{equation} 
we have 
\begin{equation}
\label{etonethird}
\sum_{\lambda\in \Lambda} c_\lambda\lambda^{1/3}=0.
\end{equation} 
Similarly, for some definition of the cubic roots~\eqref{ecubicroots} we have  ${\sum_{\lambda\in \Lambda} c_\lambda\lambda^{-5/3}=0}$.
\end{proposition}

\begin{proof}
Consider the polynomial 
$$
F(X_1,X_2,X_3)=X_1^3+X_1^3+X_1^3-3X_1X_2X_3 \in \Z[X_1,X_2,X_3]. 
$$
Write again ${\Lambda=\{\lambda_1,\lambda_2,\lambda_3\}}$. Define somehow the cubic roots ${\lambda_1^{1/3},\lambda_2^{1/3}}$ and set ${\lambda_3^{1/3}= (\lambda_1^{1/3}\lambda_2^{1/3})^{-1}}$. Now define
$$
\alpha_i= c_{\lambda_i}\lambda_i^{1/3}, \quad \beta_i= c_{\lambda_i}\lambda_i^{-5/3} \qquad (i=1,2,3). 
$$
A direct verification shows that 
$$
F(\alpha_1,\alpha_2,\alpha_3)=\sum_{\lambda\in \Lambda}c_\lambda^3\lambda-3\prod_{\lambda\in \Lambda}c_\lambda=0, 
$$
and, similarly, ${F(\beta_1,\beta_2,\beta_3)= 0 }$. 
Since ${F(X_1,X_2,X_3)}$ factors as 
$$
F(X_1,X_2,X_3)=(X_1+X_2+X_3)(X_1+\zeta X_2+\overline\zeta X_3)(X_1+\overline\zeta X_2+\zeta X_3), 
$$
where ${\zeta,\overline\zeta}$ are the primitive cubic roots of unity, the result follows. 
\end{proof}

Call ${r\in \Q}$  a \textit{rational zero} of~$T$ 
if for some definition of the rational powers 
$\lambda_1^r,\lambda_2^r,\lambda_3^r$ 
we have ${\sum_{i=1}^3c_{\lambda_i}\lambda_i^r=0}$. 

More generally, call ${r\in \Q}$  a \textit{twisted rational zero} of~$T$ if for some definition of the rational powers 
$\lambda_1^r,\lambda_2^r,\lambda_3^r$
and for some roots of unity $\xi_1,\xi_2,\xi_3$, 
we have ${\sum_{i=1}^3\xi_ic_{\lambda_i}\lambda_i^r=0}$.

We denote $\calQ_T$ the set of twisted rational zeros of~$T$. Clearly, ${\calZ_T\subset\calQ_T}$ and ${1/3,-5/3\in \calQ_T}$. It turns out that~$T$ has no other twisted rational zeros.

\begin{theorem}
\label{thtwrat}
We have 
$$
\calQ_T=\calZ_T\cup\{1/3,-5/3\}=\{0,-1,-4,-17,1/3,-5/3\}.
$$
Moreover, if  ${r\in \calQ_T}$ and the powers $\lambda_1^r,\lambda_2^r,\lambda_3^r$ are suitably defined, then for the roots of unity $\xi_1,\xi_2,\xi_3$ satisfying ${\sum_{i=1}^3\xi_ic_{\lambda_i}\lambda_i^r=0}$ we have ${\xi_1=\xi_2=\xi_3}$. 
\end{theorem}

The full proof of this theorem will appear in~\cite{LRS}. Now we will  prove only a weaker version of this theorem, addressing  twisted integral zeros. 

\begin{theorem}
\label{thtwint}
Let ${n\in \Z}$ and ${\xi_1,\xi_2,\xi_3}$ roots of unity such that ${\sum_{i=1}^3\xi_ic_{\lambda_i}\lambda_i^n=0}$. Then ${n\in \calZ_T}$ and ${\xi_1=\xi_2=\xi_3=1}$. In particular, the only twisted integral zeros of the Tribonacci sequence are its actual zeros   $0,-1,-4,-17$. 
\end{theorem}

\begin{remark}
Of course, twisted zeros can be defined for any linear recurrent sequence, not just of the Tribonacci sequence: if $U(n)$ is a linear recurrence with Binet expansion 
$$
U(n)=P_1(n)\gamma_1^n+\cdots + P_s(n) \gamma_s^n
$$
(where ${\gamma_1, \ldots, \gamma_s}$ are non-zero algebraic numbers
and $P_1, \ldots, P_s$ are polynomials with algebraic coefficients),
then we call ${r\in \Q}$ a twisted rational zero of~$U$ if for some
definition of the powers ${\gamma_1^r, \ldots, \gamma_s^r}$ and some
roots of unity ${\xi_1, \ldots, \xi_s}$ we have
${\xi_1P_1(r)\gamma_1^r+\cdots +\xi_sP_s(r)\gamma_s^r=0}$. Note that
the analogue of Theorem~\ref{thtwint} does not hold for any linear recurrent sequence. For instance,  the binary sequence ${U(n) =2^n+1^n }$ has no integral zeros, but it has a twisted zero at ${n=0}$, the relevant roots of unity being~$1$ and $-1$:
$$
1\cdot 2^0+(-1)\cdot 1^0=0. 
$$
\end{remark}

For the proof of Theorem~\ref{thtwint} we need some lemmas. 

\begin{lemma}
\label{lroots}
Let~$\alpha$ be an algebraic number of degree~$3$. Assume that $\Q(\alpha)$ is not a Galois extension of~$\Q$.
Let ${\alpha_1(=\alpha), \alpha_2, \alpha_3}$ be the conjugates of~$\alpha$ over~$\Q$. Assume further that the field ${\Q(\alpha_1,\alpha_2,\alpha_3)}$ does not contain primitive cubic roots of unity.  

Let ${\xi_1,\xi_2,\xi_3}$ be roots of unity such that 
$$
\alpha_1\xi_1+\alpha_2\xi_2+\alpha_3\xi_3=0.
$$
Then ${\xi_1=\xi_2=\xi_3}$ and hence ${\alpha_1+\alpha_2+\alpha_3=0}$. 
\end{lemma}

\begin{proof}
We may assume that ${\xi_3=1}$, so that
${\alpha_1\xi_1+\alpha_2\xi_2+\alpha_3=0}$. We want to prove that ${\xi_1=\xi_2=1}$. 

Denote ${{\mathbb K}=\Q(\alpha_1,\alpha_2,\alpha_3)}$ and
let~${\mathbb K}_0$ be the unique quadratic subfield of~${\mathbb K}$; note that~${\mathbb K}_0$ is the maximal abelian subfield of~${\mathbb K}$. 

Assume for a contradiction that
${\{\xi_1,\xi_2\}\not\subset {\mathbb K}}$.  Then there is a
non-trivial element
$\sigma\in {\Gal({\mathbb K}(\xi_1,\xi_2)/{\mathbb K})}$. We have
${\alpha_1\xi_1^\sigma+\alpha_2\xi_2^\sigma+\alpha_3=0}$ and, without
loss of generality, ${\xi_1^\sigma\ne \xi_1}$. It follows that
${\eta := \alpha_1/\alpha_2 =
  -(\xi_2-\xi_2^\sigma)/(\xi_1-\xi_1^\sigma)}$.  In particular,~$\eta$
belongs to an abelian field and so ${\eta\in {\mathbb K}_0}$.  But the
elements of~${\mathbb K}_0$ are fixed by a cyclic permutation of
${\alpha_1,\alpha_2,\alpha_3}$. Hence
${\eta= \alpha_1/\alpha_2=\alpha_2/\alpha_3=\alpha_3/\alpha_1}$. It
follows that
${\eta^3=(\alpha_1/\alpha_2)(\alpha_2/\alpha_3)(\alpha_3/\alpha_1)=1}$,
contradicting the hypothesis that $\mathbb{K}$ contain no primitive
cubic roots of unity.  We conclude that $\xi_1$ and~$\xi_2$ belong
to~${\mathbb K}$ and hence also to~${\mathbb K}_0$.

Observe that any element of $\Gal({\mathbb K}/\Q)$ of order 2
restricts to the non-trivial element $\iota$ of $\Gal({\mathbb K}_0/\Q)$.
Consider first the element of $\Gal({\mathbb K}/\Q)$ that
switches ${\alpha_1,\alpha_2}$ and fixes~$\alpha_3$.  Now we
have ${\alpha_1\xi_2^\iota+\alpha_2\xi_1^\iota+\alpha_3=0}$. If
${\xi_2^\iota\ne \xi_1}$ then
${\alpha_1/\alpha_2 = (\xi_1^\iota-\xi_2)/(\xi_1-\xi_2^\iota)\in
  {\mathbb K}_0}$, and we finish as before. Thus,
${\xi_2^\iota= \xi_1}$ (and ${\xi_1^\iota= \xi_2}$). Applying next the
element that switches ${\alpha_1,\alpha_3}$ and fixes~$\alpha_2$, we
obtain ${\alpha_1+ \alpha_2\xi_1+ \alpha_3\xi_2=0}$. Multiplying by
${\xi_1=\xi_2^{-1}}$, we get
${\alpha_1\xi_1+ \alpha_2\xi_1^2+ \alpha_3=0}$. Hence
${\alpha_2(\xi_1^2-\xi_2)=0}$, which shows that
${\xi_1^2=\xi_2=\xi_1^{-1}}$. Thus, ${\xi_1^3=1}$, which implies that
${\xi_1=1}$ by our hypothesis. Hence ${\xi_2=\xi_1^\iota=1}$ as well,
and we are done.
\end{proof}

\begin{lemma}
\label{lgam}
Let~$\lambda$ be a root of ${P(X)=X^3-X^2-X-1}$, and ${n\in \Z}$. Then ${\alpha=\lambda^n/P'(\lambda)}$ satisfies the hypothesis of Lemma~\ref{lroots}. 
\end{lemma}

\begin{proof}
Clearly, ${\alpha \in \Q(\lambda)}$, which is field of degree~$3$. If ${\Q(\alpha)\ne \Q(\lambda)}$ then ${\alpha \in \Q}$. Hence, denoting ${\lambda_1,\lambda_2,\lambda_3}$ the roots of $P(X)$, the three numbers ${\lambda_i^n/P'(\lambda_i)}$ must be equal. In particular, if~$\lambda_1$ is the real root, and $\lambda_2,\lambda_3$ are the complex conjugate roots, then
${\lambda_1^n/P'(\lambda_1)=\lambda_2^n/P'(\lambda_2)}$, which implies that 
$$
n= \frac{\log|P'(\lambda_2)/P'(\lambda_1)|}{\log|\lambda_1/\lambda_2|} = -0.718\dots\notin\Z,
$$
a contradiction.

Thus, ${\Q(\alpha)=\Q(\lambda)}$ is not a Galois extension of~$\Q$. It remains to note that its Galois closure ${\Q(\lambda_1,\lambda_2,\lambda_3)}$ may not contain primitive cubic roots of unity, because prime~$3$ is not ramified therein. 
\end{proof}

\begin{proof}[Proof of Theorem~\ref{thtwint}]
If ${\sum_{i=1}^3\xi_ic_{\lambda_i}\lambda_i^n=0}$ then the above lemmas imply that ${\xi_1=\xi_2=\xi_3}$. Hence ${T(n)=0}$, and we are done. 
\end{proof}

Theorem~\ref{thtwrat} is proved in two steps. Using a Galois-theoretic argument similar to that of Lemma~\ref{lroots}, but more involved, one reduces the problem to finding actual integral zeros of another linear recurrence, of order~$4$. Those are determined using standard technique, with logarithmic forms and Baker-Davenport reduction. See~\cite{LRS} for the details.

\section{$p$-adic analytic functions}
\label{spadic}
In this section we recall some very basic facts about $p$-adic analytic functions. Most of them are quite standard. All missing proofs, unless indicated otherwise, can be found in any standard text like~\cite{Go20}.

Let~$p$ be a prime number and let~${\mathbb K}$ be a finite extension of~$\Q_p$. We extend the standard $p$-adic absolute value ${|\cdot|}$ from~$\Q_p$ to~${\mathbb K}$, so that ${|p|_p=p^{-1}}$. We will also use the additive valuation $\nu_p$ defined by ${\nu_p(z)= - \log|z|_p/\log p}$ for ${z\in K^\times}$, with the convention ${\nu_p(0)=+\infty}$.

For ${a\in {\mathbb K}}$ and ${r>0}$ we denote $\calD(a,r)$ and $\overline\calD(a,r)$ the open and the closed disk with center~$a$ and radius~$r$:
$$
\calD(a,r)=\{z \in K: |z-a|_p<r\}, \qquad \overline \calD(a,r)=\{z \in K: |z-a|_p\le r\}.
$$ 
We denote by $\calO_{\mathbb K}$, or simply by~$\calO$ if this does not lead to a confusion, the ring of integers of~${\mathbb K}$:
$$
\calO=\{z\in {\mathbb K}: |z|_p \le 1\} = \overline\calD(0,1).  
$$
We call ${f:\calO\to\calO}$ an analytic function if there is a sequence 
${\alpha_0, \alpha_1, \alpha_2, \ldots \in \calO}$  with ${\lim_{n\to\infty}|\alpha_n|_p=0}$ such that 
$$
f(z) =\sum_{n=0}^\infty \alpha_n z^n \qquad (z\in \calO). 
$$
Note that for any ${b\in \calO}$ we have  
\begin{equation}
\label{equb}
f(z) = \sum_{k=0}^\infty \beta_k(z-b)^k, 
\end{equation} 
where 
$$ 
\beta_k=\frac{f^{(k)}(b)}{k!}= \sum_{n=k}^\infty \binom nk \alpha_n b^{n-k}.
$$

\subsection{$p$-adic order of values of an analytic function}
We start from the following trivial, but useful observation.

\begin{proposition}
\label{ptrivial}
Let $f(z)$ be an analytic function. Then for any ${a,b\in \calO}$ we have 
${|f(a)-f(b)|_p \le |a-b|_p}$. 
\end{proposition}

\begin{proof}
Substituting ${z=a}$ into~\eqref{equb} and noting that ${\beta_0=f(b)}$,  we obtain 
$$
|f(a)-f(b)|_p= |b-a|_p \left|\sum_{k=1}^\infty\beta_k|a-b|^{k-1}\right|_p. 
$$
All terms in the sum on the right belong to~$\calO$, whence the result. 
\end{proof}

Assume now that~$f$ is not identically~$0$. Then the set of zeros of~$f$ is finite, because it is a discrete subset of the compact set~$\calO$; we denote this set~$\calA$. 

\begin{theorem}
\label{thfunml}
Let~$e$ be the ramification index of $\K/\Q_p$. 
Then there exists a positive integer~$k$ such that  for every  ${i\in \{0,1, \ldots, p^k-1\}}$ we have one of the following two options.
\begin{enumerate}
\item[(C)]
There exists ${\kappa_i\in e^{-1}\Z}$ such that for ${z\in \calO}$  satisfying ${z\equiv i \pmod {p^k}}$ we have ${\nu_p(f(z)) =\kappa_i}$; in other words, $\nu_p(f(z))$ is constant on the residue class ${z\equiv i\pmod  {p^k}}$. 

\item[(L)]
There exist 
$$
a_i\in \calA, \qquad \kappa_i\in e^{-1}\Z, \qquad  \mu_i \in \Z_{>0}
$$ 
such that for   ${z\in \calO}$  satisfying ${z\equiv i \pmod {p^k}}$ we have 
$$
\nu_p(f(z)) =\kappa_i+\mu_i\nu_p(z-a_i). 
$$
\end{enumerate}

\end{theorem}

\begin{proof}
Let~$m$ be a positive integer, and for every ${j\in \{0,1, \ldots, p^m-1\}}$ define 
${f_j(z)=f(j+p^mz)}$. Clearly, if the statement holds true for every~$f_j$ then it holds for~$f$ as well. Taking~$m$ so large that every residue class ${z\equiv j\pmod{p^m}}$ contains at most one element from~$\calA$, we reduce the theorem to the case when~$f$ has at most one zero. If~$f$ does have a zero, say~$a$, then, expanding 
$$
f(z) = \alpha_\mu(z-a)^\mu+\alpha_{\mu+1}(z-a)^{\mu+1}+\cdots,
$$
with ${\mu\ge 1}$ and ${\alpha_\mu\ne 0}$, we note that the statement holds for~$f$ as soon as it holds for the analytic function ${\alpha_\mu+\alpha_{\mu+1}(z-a)+\cdots}$, which has no zero at all.

Thus, it suffices to consider the case ${\calA=\varnothing}$. We need to show that  the $p$-adic order $\nu_p(f(z))$ is constant on every residue class modulo a suitable power of~$p$.  

Since~$f$ does not vanish on~$\calO$, then, by compactness, ${|f(z)|_p}$ must be bounded from below by some strictly positive number. It follows that $f(z)$ belongs to one of the finitely many sets 
$$
\calO^\times, \pi\calO^\times, \ldots, \pi^n\calO^\times, 
$$
where~$\pi$ is a primitive element of~$\K$ and~$n$ is some positive integer. Note that ${\nu_p(\pi)=e^{-1}}$.  

Since these sets are open, their inverse images by~$f$ are open as well. Hence each of these inverse images is a union of finitely many residue classes modulo some power of~$p$. This completes the proof. 
\end{proof}

\subsection{Vanishing of power series}
\label{sshens}
In this subsection we recall two fundamental results about vanishing of a power series on~$\calO$: Hensel's Lemma and Strassman's Theorem.

Hensel's Lemma is the principal technical tool of $p$-adic
analysis. It is usually stated for polynomials, but in this article we
need a slightly more general version, for power series.

\begin{proposition}[Hensel's Lemma for power series]
\label{phens}
Let ${b_0\in \calO}$ be such that ${|f(b_0)|_p<1}$ and ${|f'(b_0)|_p=1}$. Then there exists a unique ${b\in \calO}$ such that ${f(b)=0}$  and ${|b-b_0|_p<|f(b_0)|_p}$. 
\end{proposition}

The proof can be found, for instance, in~\cite{Co}, see Theorems~8.2 and~9.4 therein, or in~\cite{Sc06}, see Theorem~{27.6} therein.

The number of zeros can be estimated using Strassman's Theorem. 

\begin{theorem}
[Strassman]
\label{thstra}
Assume that $f(z)$ does not vanish identically on $\calO$; equivalently, the coefficients $\alpha_0, \alpha_1, \ldots$ are not all~$0$. Define~$\mu$ as the largest~$m$ with the property 
$$
|\alpha_m|_p=\max\{|\alpha_n|_p: n=0,1,\ldots\}. 
$$
(Since ${|\alpha_n|_p \to 0}$, such~$\mu$ must exist.) 
Then $f(z)$ has at most~$\mu$ zeros on $\calO$. 
\end{theorem}

The proof can be found in many sources; see, for instance, \cite[Theorem~4.1]{Ca86}.

\subsection{Functions $\exp$ and $\log$ in the $p$-adic domain}
\label{ssexplog}
We denote ${\rho=p^{-1/(p-1)}}$. Let us recall the definition and the basic properties of the $p$-adic exponential and logarithmic function. 

\begin{enumerate}
\item
For ${z\in \calD(0,\rho)}$ we define  
$$
\exp(z) =\sum_{n=0}^\infty \frac{z^n}{n!}. 
$$
For ${z,w\in \calD(0,\rho)}$ we have
$$
|\exp(z)-1|_p=|z|_p, \quad \exp(z+w)=\exp(z)\exp(w), \quad \exp'(z)=\exp(z). 
$$
\item
For ${z\in \calD(1,1)}$ we define  
$$
\log(z) =\sum_{n=1}^\infty \frac{(-1)^{n-1}(z-1)^n}{n}. 
$$
For ${z,w\in \calD(1,1)}$ we have
$$
\log(zw)=\log(z)+\log(w), \quad \log'(z)=\frac1z. 
$$
\item
For ${z\in \calD(1,\rho)}$ we have 
$$
|\log(z)|_p=|z-1|_p, \quad \exp(\log(z)) =z. 
$$
\item
For ${z\in \calD(0,\rho)}$  we have 
${\log(\exp(z))=z}$. 
\end{enumerate}

\begin{remark}
\label{rrhoone}
Note that, when ${p>2}$ and~$p$ is unramified in~${\mathbb K}$, we have 
$$
\calD(0,\rho)=\calD(0,1), \qquad \calD(1,\rho)=\calD(1,1), \qquad \overline\calD(0,1)=\calD(0,p^{-1}). 
$$
This will always be the case starting from Section~\ref{sfunction}. This excludes the primes $p\in\{2,11\}$ from our analysis. 
\end{remark}

\section{$p$-adic analytic interpolation of the Tribonacci sequence}
\label{sfunction}

Recall that we denote ${\Lambda=\{\lambda_1,\lambda_2,\lambda_3\}}$ the set of roots of the polynomial 
$$
P(X)=X^3-X^2-X-1.
$$
Let~$p$ be a prime number and let ${{\mathbb K}=\Q_p(\lambda_1,\lambda_2,\lambda_3)}$ be the splitting field of $P(X)$ over~$\Q_p$.  As before, we denote~$\calO$ its ring of integers. The discriminant of $P(X)$ is $-44$. Hence, assuming in the sequel that ${p\ne 2,11}$, the field~${\mathbb K}$ is unramified over~$\Q_p$.

We denote ${d=[{\mathbb K}:\Q_p]}$. There are three possibilities. If all the roots of $P(X)$ are in~$\Q_p$ then ${{\mathbb K}=\Q_p}$ and ${d=1}$. If $P(X)$ has exactly one root in~$\Q_p$ then ${d=2}$. Finally, if $P(X)$ is irreducible in~$\Q_p$ then ${d=3}$.  

Recall that 
$$
T(n) = \sum_{\lambda\in \Lambda}c_\lambda\lambda^n, \qquad c_\lambda=\lambda P'(\lambda)^{-1}
$$
Note that, since ${p\ne 2,11}$, we have ${c_\lambda  \in \calO^\times}$ for ${\lambda \in \Lambda}$. 
Recall also that  ${T(n)=0}$ if and only if ${n\in {\mathcal Z}_T}$. 
Note that ${\Lambda \subset \calO^\times}$. Let ${N=N_p}$ be the order of the  subgroup of the multiplicative group $(\calO/p)^\times$ generated by~$\Lambda$. In \cite{ML14} this quantity is denoted $\pi(p)$. Note that 
${N\mid p^d-1}$. When $d=3$, we have the more precise divisibility relation $N\mid p^2+p+1$. 

For ${\ell\in\{0,1,\ldots, N-1\}}$ we consider the analytic function ${f_\ell:\Z_p\to\Z_p}$ defined by 
\begin{equation}
\label{efell}
f_\ell(z) = \sum_{\lambda\in \Lambda} c_\lambda\lambda^\ell\exp\bigl(z\log (\lambda^N)\bigr).
\end{equation}
Note that by the definition of~$N$ we have 
$$
\lambda^N\in \calD(1,1)=\calD(1,\rho)
$$
(see Remark~\ref{rrhoone}), so $f_\ell(z)$ is indeed well-defined for ${z\in \Z_p}$. Furthermore,  for ${m\in \Z}$ we have 
\begin{equation}
\label{efkt}
f_\ell(m) =T(\ell+mN) \in \Z.
\end{equation}
Since~$\Z$ is dense in~$\Z_p$ and~$f_\ell$ is continuous, we indeed have ${f_\ell(z)\in \Z_p}$ for ${z\in \Z_p}$. 

Note also that ${f_\ell(z)}$ does not vanish identically on~$\Z_p$: this also follows from equation~\eqref{efkt}.

\section{Proof of Theorem~\ref{ththird}} 
\label{sthfive}
We use the terminology and the notation of Section~\ref{sfunction}. In this section~$p$ is a prime number, distinct from~$2$ and~$11$, and satisfying the following two conditions: ${p\equiv 2 \pmod 3}$ and ${\Lambda \subset \Q_p}$. The last condition means that ${\K=\Q_p}$ and ${d=1}$. By the Chebotarev Density Theorem, the set of such~$p$ is infinite and, moreover, it is of density $1/12$ in the set of all primes.

We are going to show that for every such~$p$ both statements of Theorem~\ref{ththird} hold true. Actually, we will prove only the former statement: 
\begin{equation}
\label{eonethird}
n\equiv 1/3 \pmod {p-1} \Longrightarrow \nu_p(T(n)) \ge \nu_p(n-1/3), 
\end{equation}
because the second statement, with $1/3$ replaced by $-5/3$, is proved absolutely similarly.

To start with, let us make the following observation: since  ${p\equiv 2 \pmod 3}$, every element of $\Z_p^\times$ has a single cubic root in $\Z_p$. In particular, for every ${\lambda \in \Lambda}$ there is a well-defined cubic root ${\lambda^{1/3}\in \Z_p}$. 

It turns out that these cubic roots are exactly those for which we have~\eqref{etonethird}. 

\begin{proposition}
\label{ponethirdvanish}
With our choice of the cubic roots $\lambda^{1/3}$ we have 
$$
\sum_{\lambda\in \Lambda}c_\lambda\lambda^{1/3}=0.
$$
\end{proposition}

\begin{proof}
Assuming the contrary, we must have one of the options 
\begin{align}
\label{ezeta}
c_{\lambda_1}\lambda_1^{1/3}+c_{\lambda_2}\lambda_2^{1/3}+\zeta c_{\lambda_3}\lambda_3^{1/3}&=0,\\
\label{ezetazeta}
c_{\lambda_1}\lambda_1^{1/3}+\zeta c_{\lambda_2}\lambda_2^{1/3}+\zeta c_{\lambda_3}\lambda_3^{1/3}&=0,\\
\label{ebarzeta}
c_{\lambda_1}\lambda_1^{1/3}+\zeta c_{\lambda_2}\lambda_2^{1/3}+\overline\zeta c_{\lambda_3}\lambda_3^{1/3}&=0,  
\end{align}
where~$\zeta$ and~$\overline\zeta$ are the primitive cubic roots of unity. Option~\eqref{ezetazeta} reduces to~\eqref{ezeta}, so we only need to rule out~\eqref{ezeta} and~\eqref{ebarzeta}.

Note  that ${\zeta\notin \Q_p}$ because ${p\equiv 2\bmod 3}$. Therefore, by Galois conjugation,  in the case~\eqref{ezeta} we also have 
${c_{\lambda_1}\lambda_1^{1/3}+c_{\lambda_2}\lambda_2^{1/3}+\overline\zeta c_{\lambda_3}\lambda_3^{1/3}=0}$. 
Hence ${(\zeta-\overline\zeta)c_{\lambda_3}\lambda_3^{1/3}=0}$, a contradiction.

Similarly, in the case~\eqref{ebarzeta} we also have 
${c_{\lambda_1}\lambda_1^{1/3}+\overline\zeta c_{\lambda_2}\lambda_2^{1/3}+ \zeta c_{\lambda_3}\lambda_3^{1/3}=0}$. 
It follows that ${(\zeta-\overline\zeta)(c_{\lambda_2}\lambda_2^{1/3} -c_{\lambda_3}\lambda_3^{1/3})=0}$, again a contradiction.  
\end{proof}

Now we are in a position to prove~\eqref{eonethird}. We define ${N=N_p}$ as in Section~\ref{sfunction}; note that in our special case ${d=1}$ and so ${N\mid p-1}$. In particular, the residue class ${1/3\pmod N}$ is well-defined.

Let ${n\in \Z }$ and ${\ell\in \{0,\ldots, N-1\}}$ satisfy 
$$
n\equiv \ell\equiv 1/3\pmod N.
$$
We define $f_\ell(z)$ as in~\eqref{efell}.  Write ${n=\ell+Nm}$ and ${1/3=\ell+Nb}$ with ${m\in \Z}$ and ${b\in \Z_p\cap\Q}$. We have clearly 
${T(n) = f_\ell(m)}$. We claim that ${f_\ell(b)=0}$. Indeed, 
for ${\lambda\in \Lambda}$ we have 
$$
\bigl(\lambda^\ell\exp(b\log(\lambda^N)\bigr)^3 =\lambda^{3(\ell+Nb)}=\lambda.  
$$
Since ${\lambda^\ell\exp(b\log(\lambda^N)\in \Z_p}$, it must be equal to the cubic root $\lambda^{1/3}$ specified above. Hence 
\begin{equation}
\label{eflb}
f_\ell(b) = \sum_{\lambda\in \Lambda} c_\lambda\lambda^\ell\exp\bigl(b\log (\lambda^N)\bigr)=\sum_{\lambda\in \Lambda} c_\lambda\lambda^{1/3} =0. 
\end{equation}
by Proposition~\ref{ponethirdvanish}. 

Now we are done: Proposition~\ref{ptrivial} implies that 
$$
\nu_p(T(n)) = \nu_p(f_\ell(m) -f_\ell(b)) \ge \nu_p(m-b)= \nu_p(n-1/3),
$$
as wanted. Note that $p \equiv 2 \mod 3$ was only required to ensure that $3 \nmid N$. The argument above can be generalized for all $p$ such that $3 \nmid N$ and $\Lambda \subset \Q_p$.

\section{Analytic form of Conjectures~\ref{conml} and \ref{conmlrat}}
\label{sstatements}

In this section~$p$ is a prime number distinct from $2,3,11$.  We continue using the notation of Section~\ref{sfunction}. 

We are going to show that Conjectures~\ref{conml} and~\ref{conmlrat} have very natural interpretations in terms of the zeros of the functions $f_\ell(z)$.

\begin{theorem}
\label{thstatements}
\begin{enumerate}
\item
The following three statements are equivalent. 

\begin{enumerate}[ref=\textbf{(\alph*)}]
\item 
\label{imlholds}
Conjecture~\ref{conml} holds for the given~$p$. 

\item
\label{izerosn}
For every  ${\ell\in \{0, \ldots, N-1\}}$, the zeros of the function $f_\ell(z)$ belong to $N^{-1}\Z$. 

\item
\label{izt}
For every~$\ell$ the following holds: if ${b\in \Z_p}$ is a zero of $f_\ell(z)$ then ${\ell+Nb\in \calZ_T}$. 

\setcounter{jump}{\value{enumii}}
\end{enumerate}

\item
The following three statements are equivalent. 

\begin{enumerate}[ref=\textbf{(\alph*)}]
\setcounter{enumii}{\value{jump}}

\item 
\label{imlratholds}
Conjecture~\ref{conmlrat} holds for the given~$p$. 

\item
\label{izerosq}
For every ${\ell\in \{0, \ldots, N-1\}}$, the zeros of the function $f_\ell(z)$ belong to ${\Q\cap\Z_p}$. 

\item
\label{iqt}
For every~$\ell$ the following holds: if ${b\in \Z_p}$ is a zero of $f_\ell(z)$ then ${\ell+Nb\in \calQ_T}$. 
\end{enumerate}
\end{enumerate}
\end{theorem}

This theorem is very useful for producing counter-examples to both
conjectures, see Section~\ref{sex}. More importantly, it provides a
clear motivation why the conjectures can only be expected to hold for
relatively few primes. Indeed, there is absolutely no reason to expect
that every $f_\ell(z)$ would have only zeros in~$\Q$, and it is even
less of a reason to expect that it would not vanish outside a fixed
set of six elements.

Let us start with some lemmas. 

\begin{lemma}
\label{ltwisted}
If ${b\in \Q\cap \Z_p}$ is a zero of $f_\ell(z)$ then ${\ell+Nb}$ is a twisted rational zero of~$T$, as defined in Section~\ref{srat}.  
\end{lemma}

\begin{proof}
Denote ${a=\ell+Nb}$ and for every ${\lambda \in \Lambda}$ choose some determination for~$\lambda^a$. 

Let~$m$ be a non-zero integer such that ${mb\in \Z}$.  Then 
$$
\Bigl(\lambda^\ell\exp\bigl(b\log(\lambda^N)\bigr)\Bigr)^{m}= \lambda^{m\ell}\exp\bigl(mb\log(\lambda^N)\bigr)=\lambda^{ma}. 
$$
Hence ${\lambda^\ell\exp\bigl(b\log(\lambda^N)\bigr) =\xi_\lambda\lambda^a}$, where~$\xi_\lambda$ is a root of unity. 
It follows that 
$$
0=f_\ell(b) = \sum_{\lambda\in \Lambda}\xi_\lambda c_\lambda \lambda^a, 
$$
as wanted. 
\end{proof}

\begin{lemma}
\label{lcholds}
Assume that Conjecture~\ref{conmlrat} holds for a given~$p$. 
\begin{enumerate}
\item
\label{ibtoai}
Let ${\ell\in \{0,1, \ldots, N-1\}}$ and let ${b\in \Z_p}$ be a zero of $f_\ell(z)$. Then there exists ${i\in \{0,1, \ldots, Q-1\}}$ such that option~(L) holds for the residue class of~$i$, and such that ${a_i= \ell+Nb}$. In particular, ${b \in \Q}$, and if ${a_i\in \Z}$ then ${b\in N^{-1}\Z}$.

\item
\label{iaitob}
Conversely, let ${i\in \{0,1, \ldots, Q-1\}}$ be such that option~(L) holds for the residue class of~$i$. Then there exists ${\ell\in \{0,1,\ldots,  N-1\}}$   such that 
$$
f_\ell\left(\frac{a_i-\ell}N\right)=0.
$$ 
\end{enumerate}
\end{lemma}

Only item~\ref{ibtoai} will be used, but we include the converse statement for completeness. 

\begin{proof}[Proof of item~\ref{ibtoai}]
This is the argument that already appeared in the introduction. 
Let $(m_k)$ be a sequence of rational integers satisfying ${m_k\equiv b\pmod {p^k}}$, and set ${n_k=\ell+Nm_k}$. Then
$$
\nu_p(T(n_k)) = \nu_p(f_\ell(m_k)) \ge \nu_p(n_k-b) \ge k, 
$$ 
and, in particular, ${\nu_p(T(n_k)\to \infty}$ as ${k\to \infty}$. Infinitely many of the numbers~$n_k$ belong to the same residue class ${i\pmod Q}$, and we will assume in the sequel that all~$n_k$ do, by taking a subsequence.  Since ${\nu_p(T(n_k))\to \infty}$, we must have option~(L) for this residue class, and moreover, we must have ${\nu_p(n_k-a_i) \to\infty}$. Since we also have ${\nu_p(n_k-(\ell+Nb)) \to\infty}$, we obtain ${a_i=\ell+Nb}$. 
\end{proof}

\begin{proof}[Proof of item~\ref{iaitob}]
It is similar, but other way round. As in Remark~\ref{rimp}, we find a sequence of integers $(n_k)$ such that 
${n_k\equiv i \pmod Q}$ and ${n_k\equiv a_i \pmod {p^k}}$. By choosing a subsequence, we find ${\ell \in \{0,1, \ldots, N-1\}}$  that ${n_k\equiv \ell \pmod N}$ for all~$k$.

Define ${m_k = (n_k-\ell)/N}$. Then the sequence $(m_k)$ converges $p$-adically to ${(a_i-\ell)/N}$. 
Since 
${\nu_p(f_\ell(m_k))=\nu_p(T(n_k) \ge k}$, 
the sequence $(f_\ell(m_k))$ converges $p$-adically to~$0$. Hence ${f_\ell((a_i-\ell)/N)=0}$. 
\end{proof}

\begin{proof}[Proof of Theorem~\ref{thstatements}]
The implications \ref{imlholds}$\Rightarrow$\ref{izerosn}  and \ref{imlratholds}$\Rightarrow$\ref{izerosq} follow from item~\ref{ibtoai} of Lemma~\ref{lcholds}. The converse implications \ref{izerosn}$\Rightarrow$\ref{imlholds}  and \ref{izerosq}$\Rightarrow$\ref{imlratholds} follow  from Theorem~\ref{thfunml}, applied to the functions $f_\ell(z)$. Implications \ref{izerosn}$\Rightarrow$\ref{izt}  and \ref{izerosq}$\Rightarrow$\ref{iqt} follow by combining Lemma~\ref{ltwisted} with Theorems~\ref{thtwint} and~\ref{thtwrat}, respectively. Finally, the converse implications \ref{izt}$\Rightarrow$\ref{izerosn}  and \ref{iqt}$\Rightarrow$\ref{izerosq} are trivial.
\end{proof}

As a byproduct, we also established the following. 

\begin{corollary}
If Conjecture~\ref{conml} holds for the given~$p$, then the numbers $a_i$ emerging in the residue classes with option~(L) belong to the set~$\calZ_T$. If Conjecture~\ref{conmlrat} holds, then  $a_i$ belong to~$\calQ_T$. 
\end{corollary}

\section{Detecting zeros of $f_\ell(z)$}
\label{sfinding}

To make use of Theorem~\ref{thstatements}, we must develop a practical method for locating zeros of $f_\ell(z)$. 
As in the previous sections,~$p$ is  a prime number distinct from~$2$ and~$11$, and ${\ell \in \{0,1, \ldots,N-1\}}$. 

\subsection{A non-vanishing condition}

To start with, let us give a simple sufficient condition for~$f_\ell$ be non-vanishing on~$\Z_p$. 

\begin{proposition}
\label{prnonvan}
If ${p\nmid T(\ell)}$ then ${f_\ell(z)\ne 0 }$ for ${z\in \Z_p}$. 
\end{proposition}

\begin{proof}
By the definition of~$N$ we have ${f(n)\equiv f(\ell) \pmod p}$ when ${n\equiv \ell\pmod N}$. In particular, for such~$n$ we have ${|T(n)|_p=|T(\ell)|_p=1}$. In other words, for ${m\in \Z}$ we have ${|f_\ell (m)|_p=1}$. By continuity, ${|f_\ell (z)|_p=1}$ for ${z\in \Z_p}$. This completes the proof. 
\end{proof}

\subsection{The first vanishing condition}
Now let study sufficient conditions for $f_\ell(z)$ to have a zero $\Z_p$. As follows from above, the first condition must be
\begin{equation} 
\label{efirstcond}
\framebox{$p\mid T(\ell)$.}
\end{equation}
This will be assumed for the rest of the section. 

It will be more convenient to work with the function 
$$
g(z) = \frac{f_\ell(z)}{p}
$$
instead of $f_\ell(z)$ itself. 
For further use,  note that $g(z)$ has the expansion 
\begin{equation}
\label{expg}
g(z) =\sum_{k=0}^\infty \beta_k z^k
\end{equation}
with the following properties:
\begin{align}
\label{ebeinzp}
\beta_k &\in \Z_p && (k=0,1,2\ldots);\\
\label{epmidbe}
\beta_k &\in p\Z_p && (k=2,3,\ldots);\\
\label{ebetoz}
|\beta_k|_p &\to 0 && (k\to \infty). 
\end{align}
Indeed,
\begin{equation}
\label{ebetaz}
\beta_0= g(0)=\frac{f_\ell(0)}p=\frac{T(\ell)}p \in \Z
\end{equation}
by our choice of~$\ell$. Furthermore,  we have 
\begin{equation}
\label{ebek}
\beta_k =\frac{p^{k-1}}{k!}\sum_{\lambda\in \Lambda} c_\lambda\lambda^\ell\left(\frac{\log(\lambda^N)}{p}\right)^k.  
\end{equation}
Since ${\lambda^N\equiv 1\pmod p}$, we have ${\log(\lambda^N)\equiv 0\pmod p}$, which shows that the sum in~\eqref{ebek} belongs to~$\Z_p$. We also have ${p^{k-1}/k! \in \Z_p}$ when ${p\ge 3}$ and ${k\ge 1}$. Hence ${\beta_k\in \Z_p}$ for ${k\ge 1}$ as well.  This proves~\eqref{ebeinzp}. 

Next, since the sum in~\eqref{ebek} belongs to~$\Z_p$, we have
${\nu_p (\beta_k) \ge k-1-\nu_p(k!)}$. It is known that ${\nu_p(k!) <k/(p-1)}$ for ${k\ge 1}$. In particular, ${\nu_p(k!) < k/2}$ for ${p\ge 3}$. It follows that  ${\nu_p (\beta_k) >0}$ for ${k\ge 2}$ and ${\nu_p(\beta_k) \to +\infty}$ as ${k\to \infty}$. This proves~\eqref{epmidbe} and~\eqref{ebetoz}.  

Note the following consequence of~\eqref{epmidbe}: for ${z\in \Z_p}$ we have
\begin{equation}
\label{egpcgp}
g'(z) \equiv g'(0) \pmod p. 
\end{equation}
Indeed, 
$$
g'(z)= \beta_1+ \sum_{k=2}^\infty k\beta_kz^{k-1}. 
$$
Here ${\beta_1=g'(0)}$ and each term in the sum is divisible by~$p$ by~\eqref{epmidbe}.  

\subsection{The second vanishing condition}
The second condition that we impose is 
\begin{equation}
\label{egpz}
\framebox{$g'(0)\ne 0\pmod p.$}
\end{equation}
This condition means that ${\beta_1=g'(0) \in \Z_p^\times}$. 
Hence there exists ${b_0\in\Z}$   such that 
\begin{equation}
\label{edefbz}
b_0\equiv -\beta_0\beta_1^{-1}\pmod p. 
\end{equation}
Substituting ${z=b_0}$ into expansion~\eqref{expg}, and using~\eqref{epmidbe}, we obtain 
\begin{equation}
\label{epmidgm}
p\mid g(b_0). 
\end{equation}
On the other hand,~\eqref{egpcgp} and~\eqref{egpz} imply that ${g'(b_0)\equiv g'(0)\not\equiv 0\pmod p}$. 
Together with~\eqref{epmidgm} this can be expressed as 
$$
|g(b_0)|_p <1, \qquad |g'(b_0)|_p=1.
$$
Now using Hensel's Lemma as given in Proposition~\ref{phens}, we find ${b\in \Z_p}$ such that ${g(b)=0}$. Then we also have ${f_\ell(b)=0}$.   

Actually, we have even more.

\begin{proposition}
\label{prvan}
Assume that~\eqref{efirstcond} and~\eqref{egpz} hold. Then ${f_\ell(z)}$ has exactly one zero on~$\Z_p$. 
\end{proposition}

\begin{proof}
Existence of a zero is already proved above. To show uniqueness, we invoke Strassman's Theorem~\ref{thstra}. Since ${|\beta_1|_p=1}$ by~\eqref{egpz}, the quantity~$\mu$ from Theorem~\ref{thstra} must be~$1$ by~\eqref{epmidbe}. Whence the result. 
\end{proof}

\section{Sufficient conditions for validity and for failure of Con\-jec\-tures~\ref{conml} and~\ref{conmlrat}} 
\label{sex}

To implement this in practice, we need to express condition~\eqref{egpz} in terms of the Tribonacci numbers $T(n)$ rather than the function $g(z)$. This is not hard. For ${z \in p\calO}$ we have 
$$
\log z \equiv z-1 \pmod {p^2}. 
$$
In particular, for ${\lambda\in \Lambda}$ we 
$$
\frac{\log (\lambda^N)}p\equiv \frac{\lambda^N-1}{p} \pmod p.  
$$
Hence, 
\begin{equation}
\label{egprimezer}
g'(0) =\beta_1= \sum_{\lambda\in \Lambda} c_\lambda\lambda^\ell\frac{\log(\lambda^N)}{p}\equiv \sum_{\lambda\in \Lambda} c_\lambda\lambda^\ell\frac{\lambda^N-1}{p} \equiv \frac{T(\ell+N)-T(\ell)}{p} \pmod p. 
\end{equation}
Therefore condition~\eqref{egpz} is equivalent to 
\begin{equation}
\label{egpzt}
\framebox{$T(\ell+N)\not\equiv T(\ell) \pmod {p^2}.$} 
\end{equation}

Now, to disprove Conjecture~\ref{conml} for some prime number~$p$, we must find~$\ell$ such that both~\eqref{efirstcond} and~\eqref{egpzt} are satisfied, and such that the resulting zero~$b$ of $f_\ell(z)$ satisfies 
$$
\ell+bN\not\in {\mathcal Z}_T.
$$
It suffices to show that 
$$
\ell+bN\not\equiv 0, -1,-4,-17 \pmod p. 
$$
Moreover, since ${b\equiv b_0\pmod p}$, this can be re-written as 
$$
\ell+b_0N\not\equiv 0, -1,-4,-17 \pmod p. 
$$
Using~\eqref{edefbz} and~\eqref{egprimezer}, this translates into 
\begin{equation}
\label{econdbzer}
\framebox{$u:=\displaystyle \ell- \frac{T(\ell)}p \left(\frac{T(\ell+N)-T(\ell)}{p}\right)^{-1} N\not\equiv 0, -1,-4,-17 \pmod p$. }
\end{equation}
Similarly, when ${p\ne 3}$, then Conjecture~\ref{conmlrat} would fail if 
\begin{equation}
\label{econdbq}
\framebox{$u\not\equiv 0, -1,-4,-17,1/3,-5/3 \pmod p$. }
\end{equation}

Let us summarize what we proved.

\begin{theorem}
\label{thm:NoML}
Let ${p\ne 2,11}$ be a prime number, and let ${\ell \in\{0,1, \ldots N_p-1\}}$ be such that~\eqref{efirstcond},~\eqref{egpzt} and~\eqref{econdbzer} hold true. Then Conjecture~\ref{conml} fails for this~$p$. Similarly, if ${p\ne 3}$ and ~\eqref{efirstcond},~\eqref{egpzt} and~\eqref{econdbq} hold true then Conjecture~\ref{conmlrat} fails for this~$p$.\qed
\end{theorem}

Now let us give  sufficient conditions of validity of each conjecture. For Conjecture~\ref{conmlrat} we will restrict to the primes congruent to~$2$ modulo~$3$. 

\begin{theorem}
\label{thsuffml}
Let~$p$ be a prime number distinct from~$2$ and~$11$. 
Assume that for every~$\ell$ satisfying~\eqref{efirstcond}, condition~\eqref{egpzt} holds true as well, and the following also holds:  ${\ell\equiv a\pmod N}$ for some ${a\in \calZ_T}$. Then Conjecture~\ref{conml} holds for this~$p$.   
\end{theorem}

For Conjecture~\ref{conmlrat} we will restrict to the primes congruent to~$2$ modulo~$3$. 

\begin{theorem}
\label{thsuffmlrat}
Let~$p$ be a prime number satisfying $\Lambda \subset \Q_p$ and $3 \nmid N$. Assume that for every~$\ell$ satisfying~\eqref{efirstcond}, condition~\eqref{egpzt} holds true as well, and the following also holds:  ${\ell\equiv a\pmod N}$ for some ${a\in \calQ_T}$. Then Conjecture~\ref{conmlrat} holds for this~$p$.   
\end{theorem}

\begin{proof}[Proof of Theorem~\ref{thsuffml}]
Fix ${\ell\in \{0,1, \ldots, N-1\}}$. If ${p\nmid T(\ell)}$ then $f_\ell(z)$ has no zeros on $\Z_p$, see Proposition~\ref{prnonvan}. Now assume that ${p\mid T(\ell)}$. Proposition~\ref{prvan} implies that $f_\ell(z)$ has a single zero on $\Z_p$. 

Now let ${a\in \calZ_T}$ be such that ${\ell\equiv a\pmod N}$. Write ${a=\ell+Nb}$ with ${b\in \Z}$. Then 
${f_\ell(b)= T(\ell+Nb)=0}$. Thus, the single zero of $f_\ell(z)$ is~$b$.

We have just showed that condition~\ref{izt} of Theorem~\ref{thstatements} holds true for this~$p$. The theorem is proved.  
\end{proof}

The proof of Theorem~\ref{thsuffmlrat} is the same, with the exception that this time we may have ${b\notin \Z}$. However, when $\Lambda \subset \Q_p$, $p \ne 2, 11$ and due to $3 \nmid N$, we still have ${f_\ell(b)=0}$, see~\eqref{eflb}.

\section{The proofs of Theorems \ref{thm:p=3}, \ref{thm5to283}, and \ref{thm5to283rat}} 
\label{sec:comput}

We start with the negative part (part (i)) of Theorem \ref{thm5to283}.  We implemented the algorithms implied by Theorem \ref{thm:NoML} in Mathematica for all primes $p\le 600$. There are $109$ primes $p\le 600$. For each prime $p$, we first computed $N:=N_p$, the period of $(T_n)_{n\in {\mathbb Z}}$ 
modulo $p$. Then for each $p$ we searched $\ell$ such that ~\eqref{efirstcond},~\eqref{egpzt} and~\eqref{econdbzer} all hold true. This calculation took a few minutes and found such an example $\ell$ for all 
$p\le 600$ except for $p\in \{2,3,11, 83,103,163,397\}$. See Table~\ref{taells} for the actual data. This proves the negative part of Theorem \ref{thm5to283}.
\begin{table}
\caption{Data for the proofs of Theorems ~\ref{thm5to283} and ~\ref{thm5to283rat}. A $^*$ means that for this prime, Theorem ~\ref{thm5to283rat} does not conclude.}
\label{taells}
\medskip
{\scriptsize{
\begin{tabular}{|c|c|c|c||c|c|c|c||c|c|c|c|}
\hline $p$ & $N$ & $\ell$ & $u$ &  $p$ & $N$ & $\ell$ & $u$ & $p$ & $N$ & $\ell$ & $u$\\  \hline
5 & 31 & 21 & 2 & 179 & 32221 & 100 & 114 & 379 & 48007 & 309 & 76 \\ \hline
7 & 48 & 5 & 1 & 181 & 10981 & 25 & 66 & 383 & 147073 & 219 & 338 \\ \hline
13 & 168 & 6 & 4 & 191 & 36673 & 72 & 22 & 389 & 151711 & 1739 & 354 \\ \hline
17 & 96 & 28 & 7 & 193 & 4656 & 171 & 76 & $401^*$ & 400 & 265 & 132 \\ \hline
19 & 360 & 18 & 12 & 197 & 3234 & 382 & 84 & 409 & 41820 & 365 & 310 \\ \hline
23 & 553 & 29 & 15 & 199 & 198 & 26 & 40 & $419^*$ & 418 & 277 & 138 \\ \hline
29 & 140 & 77 & 24 & 211 & 5565 & 83 & 203 & 421 & 420 & 118 & 214 \\ \hline
31 & 331 & 14 & 22 & 223 & 16651 & 361 & 38 & 431 & 61920 & 465 & 51 \\ \hline
37 & 469 & 19 & 17 & 227 & 17176 & 34 & 57 & 433 & 62641 & 385 & 334 \\ \hline
41 & 560 & 35 & 15 & 229 & 17557 & 249 & 61 & 439 & 6424 & 781 & 160 \\ \hline
43 & 308 & 82 & 11 & 233 & 9048 & 36 & 126 & 443 & 196693 & 516 & 21 \\ \hline
$47^*$ & 46 & 31 & 16 & 239 & 4760 & 28 & 85 & 449 & 202051 & 107 & 229 \\ \hline
$53^*$ & 52 & 33 & 16 & 241 & 29040 & 506 & 57 & 457 & 34808 & 858 & 30 \\ \hline
59 & 3541 & 64 & 34 & 251 & 63253 & 304 & 218 & 461 & 35420 & 192 & 9 \\ \hline
61 & 1860 & 68 & 34 & 257 & 256 & 54 & 34 & 463 & 71611 & 624 & 199 \\ \hline
67 & 1519 & 100 & 43 & 263 & 23056 & 37 & 214 & 467 & 218557 & 1269 & 70 \\ \hline
71 & 5113 & 132 & 62 & $269^*$ & 268 & 177 & 88 & 479 & 76480 & 56 & 8 \\ \hline
73 & 5328 & 31 & 30 & 271 & 73440 & 331 & 165 & 487 & 79219 & 131 & 85 \\ \hline
79 & 3120 & 18 & 76 & 277 & 12788 & 61 & 191 & 491 & 10045 & 802 & 289 \\ \hline
89 & 8011 & 109 & 8 & 281 & 13160 & 536 & 62 & $499^*$ & 166 & 109 & 331 \\ \hline
97 & 3169 & 19 & 51 & 283 & 13348 & 777 & 193 & 503 & 42168 & 107 & 497 \\ \hline
101 & 680 & 186 & 23 & 293 & 28616 & 458 & 200 & 509 & 259591 & 1228 & 433 \\ \hline
107 & 1272 & 184 & 52 & 307 & 31416 & 30 & 163 & 521 & 271963 & 2058 & 220 \\ \hline
109 & 990 & 105 & 62 & 311 & 310 & 123 & 58 & 523 & 273528 & 237 & 16 \\ \hline
113 & 12883 & 172 & 15 & 313 & 32761 & 29 & 184 & 541 & 58536 & 633 & 200 \\ \hline
127 & 5376 & 586 & 30 & 317 & 100807 & 36 & 186 & 547 & 149604 & 104 & 72 \\ \hline
131 & 5720 & 79 & 101 & 331 & 36631 & 188 & 4 & 557 & 103416 & 509 & 424 \\ \hline
137 & 18907 & 11 & 5 & 337 & 16224 & 320 & 103 & 563 & 52828 & 87 & 232 \\ \hline
139 & 3864 & 34 & 49 & 347 & 40136 & 156 & 244 & 569 & 53960 & 322 & 49 \\ \hline
149 & 7400 & 10 & 38 & 349 & 17400 & 1428 & 33 & 571 & 40755 & 527 & 155 \\ \hline
151 & 2850 & 223 & 142 & 353 & 124963 & 95 & 38 & 577 & 111169 & 361 & 85 \\ \hline
157 & 8269 & 71 & 107 & 359 & 42960 & 1204 & 115 & $587^*$ & 293 & 96 & 194 \\ \hline
167 & 9296 & 41 & 68 & 367 & 45019 & 692 & 99 & 593 & 3256 & 849 & 422 \\ \hline
173 & 2494 & 314 & 25 & 373 & 139128 & 279 & 188 & 599 & 598 & 257 & 485 \\ \hline

  \end{tabular}
 }}
\end{table}
 
As for part (ii) of Theorem \ref{thm5to283}, when $p\in\{83,397\}$, we have that $N=N_p$ is $287$ and $132$, respectively. In  both cases, the only $\ell\in \{0,1,\ldots,N-1\}$ such that $T(\ell)\equiv 0\pmod p$ are 
$\ell\equiv -17,-4,-1,0\pmod N$. Furthermore for $\ell\in {\mathcal Z}_T$, we have $(T(N+\ell)-T(\ell))/p=T(N+\ell)/p\not\equiv 0\pmod p$. Thus, taking $\ell\in  {\mathcal Z}_T$ and writing for positive integers $n\equiv \ell\pmod N$, $z=(n-\ell)/N$,  
we have that 
$$
T(n)=f_{\ell}(z)=pg(z)=p\sum_{k\ge 0} \beta_k z^k.
$$
Note that $\beta_0=g(0)=T(\ell)/p=0$, and $\beta_1=g'(0)\equiv T(N+\ell)/p\pmod p$, so $|\beta_1|_p=|g'(0)|_p=1$. Further, since $\nu_p(p^{k-1}/(k-1)!)\ge 1$ for all $k\ge 2$, it follows that $|\beta_k|_p<1$ for $k\ge 2$. This shows that 
$$
\nu_p(T(n))=1+\nu_p(g(z))=1+\nu_p\left(\sum_{k\ge 1} \beta_k z^k\right)=1+\nu_p(z)=1+\nu_p(n-\ell),
$$
which proves Item (ii) of Theorem \ref{thm5to283}. 

Theorem \ref{thm5to283rat} is proved similarly. Only ~\eqref{econdbzer} is exchanged for ~\eqref{econdbq} and $\calZ_T$ for $\calQ_T$.

\begin{proof}[Proof of Theorem \ref{thm:p=3}]
For $p=3$, we have $N=13$. The only $\ell\in \{0,\ldots,12\}$ such that $T(\ell)\equiv 0\pmod 3$ are $\ell\in \{0,7,9,12\}$. When $\ell=7$, the subsequence $T(13n+\ell)$ is constant 6 modulo 9, and so $v_3(T(n)) = 1$ if $n \equiv 7 \mod 13$.

Next assume that $\ell= 0,~-1$. Then $g(0)$ is congruent modulo $3$ to one of $T(13)/3,~T(12)/3$ and they are both $0$, so we need additional terms. We have
\begin{align*}
\beta_1 & =  \sum_{\lambda \in \Lambda} c_{\lambda} \lambda^{\ell} \left(\frac{\log \lambda^{N}}{3}\right)\\
& \equiv   \sum_{\lambda \in \Lambda} c_{\lambda}\lambda^{\ell}\left(\frac{\lambda^{N}-1}{3}-\frac{(\lambda^{N}-1)^2}{2\cdot 3}\right)&&\pmod {3^2}\\
& \equiv   \frac{T(N+\ell)-T(\ell)}{3}-\frac{T(2N+\ell)-2T(N+\ell)+T(\ell)}{2 \cdot 3}&&\pmod {3^2}.
\end{align*}
For both $\ell = 0,~-1$, we have $\nu_3((T(N+\ell)-T(\ell)))=2$ and 
$$
\nu_3(T(2N+\ell)-2T(N+\ell)+T(\ell))=3.
$$ 
Thus, $\nu_3(\beta_1)=1$. 
For $j\ge 4$, we get that $\nu_3(\beta_j) \ge \nu_3(3^{j-1}/j!)\ge 2$. It remains to study $\nu_3(\beta_j)$ for $j=2,3$. But we have 
\begin{align*}
\beta_j & =  \frac{3^{j-1}}{j!} \sum_{\lambda \in \Lambda} c_{\lambda} \lambda^{\ell} \left(\frac{\log \lambda^{N}}{3}\right)^j\\
& \equiv  \frac{3^{j-1}}{j!}  \sum_{\lambda \in \Lambda} c_{\lambda}\lambda^{\ell}\left(\frac{\lambda^{N}-1}{3}\right)^j&&\pmod {3^j}\\
& \equiv   \frac{3^{j-1}}{j!} \left(\frac{\sum_{i=0}^j (-1)^{j-i} \binom{j}{i} T((j-i)N+\ell)}{3^j}\right)&&\pmod {3^j},
\end{align*}
and computations show that for $j=2,3$, we have 
\begin{align*}
\nu_3(T(2N+\ell)-2T(N+\ell)+T(\ell)) & =  3;\\
\nu_3(T(3N+\ell)-3T(2N+\ell)+3T(N+\ell)-T(\ell)) & =  5.
\end{align*}
Since also $\nu_3(3^{j-1}/j!)=1$  for $j=2,3$, we get that $\nu_3(\beta_2)\ge 2,~\nu_3(\beta_3)\ge 2$. Thus, for $n\equiv \ell\pmod {13}$, we have
$$
\nu_3(T(n))=\nu_3\left(\beta_1z+\sum_{k\ge 2} \beta_k z^k\right)=\nu_3(\beta_1 z)=2+\nu_3(n-\ell). 
$$ 
It remains to study the case $\ell=9$. For this, we take $N_1=3N=39$. This case then becomes $\ell\equiv -4,-17,9\pmod {39}$. For $\ell = 9$, the subsequence $T(3Nn + \ell)$ is constantly $3^4 \mod 3^5$, and so $v_3(T(n)) = 4$ for all $n \equiv 9 \mod 3N$. So let $\ell = -4, -17$. Then, if $n\equiv \ell\pmod {3N}$, putting $z=(n-\ell)/3N$, we get 
$$
T(n)=3^2g(z),
$$
where now 
$$
g(z)=\sum_{\lambda\in \Lambda} c_{\lambda} \lambda^{\ell} \left(\frac{\exp(\log \lambda^{3Nz})}{3^2}\right)=\sum_{k\ge 0} \beta_k z^k.
$$
For both possibilities of $\ell$, $\beta_0 = 0$. However, modulo $3^4$ we have
\begin{align*}
\beta_1 & =  \sum_{\lambda\in \Lambda} c_{\lambda} \lambda^{\ell} \left(\frac{\log \lambda^{3N}}{3^2}\right)\\
& \equiv  \frac{1}{3^2}\sum_{\lambda\in \Lambda}c_{\lambda} \lambda^{\ell} \left(\lambda^{3N}-1-\frac{(\lambda^{3N}-1)^2}{2}\right)\\
 & \equiv  \frac{1}{3^2}\left(T(3N+\ell)-T(\ell)-\frac{T(2\cdot 3N+\ell)-2T(3N+\ell)+T(\ell)}{2}\right).
\end{align*}
In both cases, $\nu_3(T(3N+\ell) - T(\ell)=5$ but $\nu_3(T(2\cdot N+\ell)-2T(3N+\ell) + T(\ell))=6$. Thus, $\nu_3(\beta_1) = 3$. Since $v_3(\beta_j) \ge \nu_3(3^{2(j-1)}/j!)\ge 4$ for $j\ge 4$, we only need to calculate $\beta_2$ and $\beta_3$. We find that 
$$
\beta_2\equiv \frac{3^2}{2!} \left(\frac{T(2\cdot 3N+\ell)-2T(3N+\ell)+T(\ell)}{3^4}\right)\equiv 0\pmod {3^4}
$$
and 
$$
\beta_3\equiv \frac{3^4}{3!} \left(\frac{T(3 \cdot 3N+\ell)-3T(2\cdot 3 N+\ell)+3T(3N+\ell)-T(\ell)}{3^6}\right) \equiv 0\pmod {3^6}.
$$
We conclude that 
\begin{align*}
\nu_3(T(n)) & =  \nu_3\left(3^2\left(\sum_{k\ge 0} \beta_k z^k\right)\right)\\
&=\nu_3(3^2\beta_0z^k)\\
& =  5+\nu_3((n+\ell)/39)\\
&=4+\nu_3(n+\ell),
\end{align*}
which completes the proof of this theorem.
\end{proof}

For the primes $p\in \{11,103,163\}$ not covered by Theorem
\ref{thm5to283}, as we previously said, our method does not handle
$11$. As for $p\in\{103,163\}$, a computer calculation found that for such
primes whenever $\ell\in \{0,1,\ldots,N-1\}$ is such that condition
\eqref{efirstcond} is satisfied, then \eqref{egpzt} holds but
\eqref{econdbzer} fails, so our method could not conclude. Similarly,
for the primes $p \in \{11,47,53,103,163\}$ which are not covered by
Theorem \ref{thm5to283rat}, our methods fail. Again, $11$ is excluded
and for $p \in \{103, 163\}$ we cannot conclude for the same
reason. Moreover, $p  \in \{47, 53\}$ suffer the problem that modulo $N$, $-17$ is congruent to either $1/3$ or $-5/3$, and so congruences modulo $p$ are too weak to conclude.

\section{Conjectures and Heuristics}
\label{sheur}
Let ${\mathcal ML}$ and ${\mathcal NMLR}$ be the subsets of primes $p$ such that Conjecture \ref{conml} holds and Conjecture \ref{conmlrat} fails, respectively. We propose the following conjecture.

\begin{conjecture} Both subsets ${\mathcal ML}$ and ${\mathcal NMLR}$ are infinite. In fact, they are both of positive lower density as subsets of the set of all primes.
\end{conjecture}

We conclude by offering some heuristics to support our conjecture. Let $k$ be a large positive integer. The splitting field of the polynomial
$$
g(X)=f(X^k)=X^{3k}-X^{2k}-X^k-1
$$
is ${\mathbb L}_k={\mathbb Q}(\sqrt[k]{\alpha},\sqrt[k]{\beta},\zeta_k)$, where $\zeta_k$ is some primitive root of unity  of order $k$. The degree of ${\mathbb L}_k$ is at most $k^2\phi(k)$, where $\phi(k)$ is the Euler function of $k$. By the Chebotarev Density Theorem, the primes such that $p\equiv 1\pmod k$ and also $\alpha^{(p-1)/k}\equiv \beta^{(p-1)/k} \equiv 1\pmod p$ form a set of density which is at least $1/(k^2\phi(k))$. For such primes, $N\mid (p-1)/k$, so $N$ is small. Since $N\le (p-1)/k$ a proportion of only about $1/k$ residues modulo $p$ (at most) are in the image of $\{T_\ell\pmod p: 0\le \ell\le N-1\}$ which suggests that the probability of having an additional  zero modulo $p$; i.e., a positive integer $\ell$ such that $T_\ell\equiv 0\pmod p$ and $\ell\not\equiv 17,-4,-1,0\pmod p$ should be at most $1/k$. Thus, for a positive proportion of such primes, maybe at least $(k-1)/(k^3\phi(k))$ of them, 
have the property that $T_\ell\equiv 0\pmod p$ implies $\ell\equiv -17,-4,-1,0\pmod p$. For such primes, $f_{\ell}(0)=0$, so 
$$
f_{\ell}z=p\sum_{k=1}^{\infty} \beta_k z^k.
$$
If additionally \eqref{egpzt} is satisfied, so $\beta_1\not\equiv 0\pmod p$, which we conjecture it happens for most such primes, then we would get that $\nu_p(T_n)=0$ provided $n\not\equiv -17,-4,-1,0\pmod p$ and 
$\nu_p(T_n)=1+\nu_p(n-c)$ for $n\equiv c\pmod {p}$, with $c\in \{-17,-4,-1,0\}$. This heuristic suggests that ${\mathcal ML}$ is infinite and of positive lower density.

For ${\mathcal NMLR}$ let $p$ be a prime such that $p\equiv 2\pmod 3$ and $f(X)\pmod p$ is irreducible. By the Chebotarev Density Theorem the set of such primes has density $1/6$. For them 
$N\mid p^2+p+1$. Let $P(m)$ be the largest prime factor of the positive integer $m$. For each fixed $u\in (0,1)$, the positive integers $n$ such that $P(n)\le n^{u}$ are called 
{\it smooth}. It is known that the set of smooth numbers has a density $\rho(u)$, where $\rho$ is  the Dickman function. It is conjectured that numbers of the form 
$g(p)$ where $g(X)$ is some irreducible polynomial should behave like random integers with respect to smoothness and in particular that $P(g(p))>g(p)^{u}$ should hold for a positive proportion of primes $p$, but this has only been proved for linear polynomials $g(X)$ and values of $u$ not very close to $1$ (for example, Fouvry \cite{Fou} proved that  for any nonzero integer $a$ the inequality $P(p-a)>p^{0.67}$ holds for a positive proportion of primes $p$). So, let us assume that there is a positive proportion of primes $p$ such $f(X)$ is irreducible modulo $p$ and $P(p^2+p+1)>p^{1.6}$. Let $p$ be such a prime and let $q=P(p^2+p+1)$. Then $N\mid p^2+p+1$. If 
$q\nmid N$, then $N\mid (p^2+p+1)/q<p^{0.4}$. However, an argument of Erd\Humlaut{o}s and Murty from \cite{EM} 
shows that for any positive real number $X$ the number of primes $p\le X$ which divide $N_{{\mathbb K}/{\mathbb Q}} (\alpha^k)$ for some $k\le X^{0.4}$ is $O(X^{0.8})$ 
which is $o(\pi(X))$ as $X\to\infty$. This shows that for most of our primes $p$ (namely, $p\equiv 2\pmod 3$, $f(X)\pmod p$ is irreducible and $P(p^2+p+1)>p^{1.6}$), we have that $q\mid N$. In particular, $N>p^{1.6}$. Now Theorem 7.2 in \cite{EPSW} tells us that 
$$
\#\{0\le \ell\le N-1: T_\ell\equiv 0\pmod p\}=\frac{N}{p}+O(p^{1/2})=(1+o(1))\frac{N}{p}.
$$
Thus, there are many $\ell$ in $[0,N-1]$ with $T_{\ell}\equiv 0\pmod p$. Of these not all might create $p$-adic zeros since for example, it might happen that $(T_{N+\ell}-T_{\ell})/p\equiv 0\pmod p$, or even if this number is nonzero modulo $p$,  it might be that \eqref{econdbq} is not satisfied. However, since we have no reason to believe that the above numbers are anything but random modulo $p$, we assume that the first condition fails with probability $1/p$ and the second one fails with probability $6/p$, getting in this way that the number of $\ell\in [0,N-1]$
such that $\ell\not\equiv -17,-4,-1,0\pmod p$ and both conditions \eqref{egpzt} and \eqref{econdbq} hold is $(1+o(1))N/p+O(N/p^2)=(1+o(1))N/p$. So, for most of such primes Conjecture \ref{conmlrat}
would fail, which suggests that ${\mathcal NMLR}$ is of positive lower density. 

\section*{Acknowledgements}

We thank Keith Conrad for helpful advice. Yu. B. and F. L. worked on this paper during a visit at the MPI-SWS in Spring 2022. These authors thank this institution for hospitality and support. Yu. B. was also supported in part by the ANR project JINVARIANT. 

\footnotesize{

\bibliographystyle{amsplain}
\bibliography{tribo}
}

\end{document}